\documentclass[12pt,reqno]{amsart}
\usepackage{amsopn,amssymb,mathrsfs,mathrsfs,a4wide,color,url,microtype}
\usepackage[unicode,linktocpage]{hyperref}
\usepackage{paralist}
\usepackage{fancyvrb}
\usepackage{amscd}

\newcommand{\bib}{\bibitem}

\def\d{\operatorname d}
\def\Z{\mathbb{Z}}

\def\vep{\varepsilon}

\def\sup{\operatorname{sup}}

\def\sgn{\operatorname{sgn}}

\def\f{\mathcal F}
\def\fii{\varphi}
\def\l{\ell}

\def\ov{\overline}
\def\sub{\subseteq}

\def\{{\left\lbrace}
\def\}{\right\rbrace}
\def\d{\delta}

\def\vvs{\vskip 2 mm}
\def\vs{\vskip 5 mm}

\def\Lip{\text{Lip}}
\def\n{\mathcal N}
\def\m{\mathcal M}
\def\lip{\operatorname{Lip}}

\def\a{\alpha}

\def\b{\beta}

\def\F{\mathcal F}

\def\n{\|}

\def\({\big(}
\def\){\big)}
\def \bo {\Big(}
\def \bc {\Big)}
\def\({\bigl(}
\def\){\bigr)}

\newtheorem{definition}{Definition}
\newtheorem{theorem}{Theorem}
\newtheorem{lemma}[theorem]{Lemma}
\newtheorem{proposition}[theorem]{Proposition}
\newtheorem{corollary}[theorem]{Corollary}
\theoremstyle{remark}
\newtheorem*{remark*}{Remark}

\newcommand\absb[2]{\csname#1l\endcsname|#2\csname#1r\endcsname|}

\newcommand\normb[2]{\csname#1l\endcsname\|#2\csname#1r\endcsname\|}

\newcommand\N{\mathbb N}
\newcommand\R{\mathbb R}

\DeclareMathOperator{\dist}{dist}

\allowdisplaybreaks

\begin{document}

\title{Some remarks on the structure of Lipschitz-free spaces}

\author{Petr H\'ajek}
\address{Mathematical Institute\\Czech Academy of Science\\\v Zitn\'a 25
\\115 67 Praha 1\\
Czech Republic\\
and Department of Mathematics\\Faculty of Electrical Engineering\\
Czech Technical University in Prague\\ Zikova 4, 160 00, Prague}
\email{hajek@math.cas.cz}

\author{Mat\v ej Novotn\'y}
\address{Department of Mathematics\\Faculty of Electrical Engineering\\
Czech Technical University in Prague\\ Zikova 4, 160 00, Prague}
\email{novotny@math.feld.cvut.cz}


\thanks{The work was supported in part by GA\v CR 16-073785, RVO: 67985840 and by grant SGS15/194/OHK3/3T/13 of CTU in Prague.}

\begin{abstract}
We give several structural results concerning the  Lipschitz-free spaces
$\mathcal F(M)$, where $M$ is a metric space. We show that $\mathcal F(M)$
contains a complemented copy of $\ell_1(\Gamma)$, where $\Gamma=\text{dens}(M)$.
If $\mathcal N$ is a net in a finite dimensional Banach space $X$, we
show that $\mathcal F(\mathcal N)$ is isomorphic to its square.
If $X$ contains a complemented copy of $\ell_p, c_0$ then $\f(\mathcal N)$
is isomorphic to its $\ell_1$-sum. Finally, we prove that
for all $X\cong C(K)$ spaces, where $K$ is a metrizable
compact,  $\f(\mathcal N)$ are mutually isomorphic
spaces with a Schauder basis.
\end{abstract}

\maketitle

\section{Introduction}

Let $(M,d)$  be a metric space  and $0\in M$ be a distinguished point.
The triple $(M,d,0)$ is called pointed metric space.
By $\Lip_0(M)$ we denote the Banach space of all Lipschitz real valued
functions $f:M\to\R$, such that $f(0)=0$. The norm of
$f\in\Lip_0(M)$ is defined
as the smallest Lipschitz constant $L=\Lip(f)$ of $f$, i.e.

$$\Lip(f)=\sup\{\frac{|f(x)-f(y)|}{d(x,y)},\ x,y\in M,\
x\neq y\}.$$

 The Dirac map
$\delta : M \to \text{Lip}_0(M)^*$ defined by
$\langle f, \delta(p)\rangle =f(p)$
for $f\in \text{Lip}_0(M)$ and $p\in M$ is an isometric
embedding from $M$ into $\text{Lip}_0(M)^*$.
Note that $\delta(0)=0$. The closed linear span of
$\{\delta(p),\ p\in M\}$ is denoted
$\mathcal F (M)$ and called the Lipschitz-free space over $M$
(or just free space, for short). Clearly,
\[
\|m\|_{\f(M)}=\sup\{\langle m,f\rangle: f\in\Lip_0(M), \|f\|\le1\}
\]

It follows from the compactness of the unit ball of $\text{Lip}_0(M)$ with
respect to the topology of
pointwise convergence, that $\mathcal F (M)$ can be seen as the canonical
predual of $\text{Lip}_0(M)$, i.e. $\f(M)^*=\Lip_0(M)$ holds isometrically
(\cite{W} Chapter 2 for details).
\vvs

Lipschitz free spaces have gained importance in the non-linear
structural theory of Banach spaces after the appearance of the
seminal paper \cite{GK} of Godefroy and Kalton, and the subsequent
work of these and many other authors e.g. \cite{K}, \cite{K2}, \cite{K3},
\cite{K4}, \cite{GO}, \cite{LP}, \cite{HP}, \cite{HLP}, \cite{PS}, \cite{Kauf},
\cite{CDW}, \cite{CD}, \cite{DF}, \cite{CJ}, \cite{NS}
\cite{D}, \cite{D2}, \cite{DKP}. Free spaces can be used efficiently
for constructions of various examples of Lipschitz-isomorphic Banach
spaces $X, Y$ which are not linearly isomorphic.
To this end, structural properties of their free spaces $\f(X)$, as well
as free spaces of their subsets, enter the game. For example,
in the separable setting,
$\f(X)$ contains a complemented copy of $X$ \cite{GK},
and it is isomorphic to its
$\ell_1$-sum. On the other hand,
if $\mathcal N$ is a net in $X$ then $\f(\mathcal N)$
is a Schur space \cite{K2} and it has the approximation property.

\vvs

A comprehensive background
on free spaces of metric spaces can be found in the book of Weaver \cite{W}.
There are several surveys exposing the applications of this notion
to the nonlinear structural theory of Banach spaces, in particular
\cite{K}, \cite{GLZ}.

\vs

Our first observation in this note is that $\f(M)$  contains a
complemented copy of $\ell_1(\Gamma)$, where $\Gamma$ is the density
character of an arbitrary infinite
metric space $M$. Our proof could be adjusted also
to the case $\Gamma=\omega_0$,
which is one of the main results in \cite{CDW}.

The main purpose of this note is to prove several structural
results, focusing mainly on the case when $M$ is a uniformly
discrete metric space, in particular a net $\mathcal N$ in a Banach space $X$.
Our results run parallel (as we have realized during the preparation
of this note) to those of Kaufmann \cite{Kauf}, resp. Dutrieux and Ferenczi
\cite{DF} which are concerned with the bigger (in a sense) space $\f(X)$.
However, the space $\f(\mathcal N)$ is only the linear quotient of $\f(X)$,
so the results are certainly not formally transferable.
In particular, the discrete setting prohibits the use of the "scaling towards zero" arguments
(used e.g. in \cite{Kauf}),
which leads to complications in proving that our free spaces
are linearly isomorphic to their squares, or even $\ell_1$-sums.
We are able to show these facts at least for nets in finite
dimensional Banach spaces and all classical Banach spaces.
Surprisingly, the proofs for the finite dimensional
case and the infinite dimensional case  are rather different.

Our main technical result is that $\f(\mathcal N)$ has a Schauder
basis for all nets in $C(K)$ spaces, $K$ metrizable compact. The
constructive proof is obtained in $c_0$, and the result is then transferred
into the $C(K)$ situation by using the abstract theory developed in the
first part of our note.

\vs

Let us start with some definitions and preliminary results.
Let $N\subset M$ be metric spaces, and assume that the distinguished point
$0\in M$ serves as a distinguished point in $N$ as well. Then the identity
mapping leads to the canonical isometric embedding
$\mathcal F(N)\hookrightarrow\mathcal F(M)$ (\cite{W} p.42). In order to
study the complementability properties of this subspace, one can rely on
the theory of quotients of metric space, as outlined in \cite{W} p.11
or \cite{Kauf}.
For our purposes we will outline an alternative (but equivalent)
description of the situation.

\begin{definition}
Let $N\subset M$ be metric spaces, $0\in N$. We denote by
\[
Lip_N(M)=\{f\in Lip_0(M): f|_N=0\}.
\]
\end{definition}

It is clear that $Lip_N(M)$ is a closed linear subspace of $Lip_0(M)$,
which is moreover $w^*$-closed. Indeed,
by the general perpendicularity principles (\cite{FHHMZ} p.56) we obtain

\[
Lip_N(M)=\mathcal F(N)^\perp,\;\; \mathcal F(N)=Lip_N(M)_\perp
\]
Hence there is a canonical isometric isomorphism
\[
Lip_N(M)\cong (\mathcal F(M)/\mathcal F(N))^*
\]

Since the space of all finite linear combinations of Dirac functionals
is linearly dense in $\f(M)$, resp.  also in $\f(N)$, it is clear
that the image of finite linear combinations of Dirac functionals supported
outside the set $N$, under the quotient mapping
$\f(M)\to\f(M)/\f(N)$ is linearly dense. Moreover, it is nonzero
for nontrivial combinations.

\begin{definition}
If $\mu=\sum_{j=1}^n a_j\delta_{t_j}: t_j\in M\setminus N$ then we
let
\[
\|\mu\|_{\mathcal F_N(M)}=\sup\langle\mu, f\rangle,\;\;f\in Lip_N(M),
\;\|f\|\le1.
\]

\[
\mathcal F_N(M)=\overline{\{\mu=\sum_{j=1}^n a_j\delta_{t_j}:
t_j\in M\setminus N\}}^{\|\cdot\|_{\mathcal F_N(M)}}.
\]
i.e. we complete the space of finite sums of Dirac functionals with respect to
the duality
\[
\langle \mathcal F_N(M), Lip_N(M)\rangle.
\]
\end{definition}

Clearly, our definition gives an isometric isomorphism

\[
\mathcal F_N(M)\cong \mathcal F(M)/\mathcal F(N)
\]

\begin{proposition}\label{ret}

Let $N\subset M$ be metric spaces. If there exists a Lipschitz
retraction $r:M\to N$ then

\[
\mathcal F(M)\cong \mathcal F(N)\oplus\mathcal F_N(M).
\]
\end{proposition}

This follows readily from the
alternative description using metric quotients  (e.g. in \cite{Kauf} Lemma 2.2)
using the fact that $\mathcal F_N(M)\cong\mathcal F(M/N)$.

We say that the metric space $(M,d)$ is $\delta$-uniformly discrete if
there exists $\delta>0$ such that $d(x,y)\ge\delta, x,y\in M$.
The metric space is uniformly discrete if it is $\delta$-uniformly
discrete for some $\delta>0$.

If $\alpha, \beta>0$ we say that a subset $N\subset M$ is
a $(\alpha,\beta)$-net in $M$ provided it is $\alpha$-uniformly
discrete and $d(x,N)<\beta, x\in M$.

It is easy to see that every maximal $\delta$-separated subset $N\subset M$,
which exists due to the Zorn maximal principle,
is automatically a $(\delta,\delta+\vep)$-net, for any $\vep>0$.

\begin{proposition}\label{l-1-sum}
Let $(M,d,0)$ be a pointed
metric space, $K>0$,  $\{M_\alpha\}_{\alpha\in\Gamma}$
be a system of pairwise disjoint subsets of $M$, and
$0\in N\subset M\setminus\cup_{\alpha\in\Gamma}M_\alpha$.
Suppose that for all $\beta\in\Gamma$ and all $x\in M_{\beta}$ holds
\[
d(x,\cup_{\alpha\in\Gamma, \alpha\ne\beta}M_\alpha)\ge Kd(x,N).\;\;
\]
Then
\[
\mathcal{F}_N(N\cup\cup_{\alpha\in\Gamma}M_\alpha)\cong
(\oplus_{\alpha\in\Gamma}
\mathcal{F}_N(N\cup M_\alpha))_{\ell_1(\Gamma)}.
\]
In particular, if $N=\{0\}$ then
\[
\mathcal{F}(\{0\}\cup\cup_{\alpha\in\Gamma}M_\alpha)\cong(\oplus_{\alpha\in\Gamma}
\mathcal{F}(\{0\}\cup M_\alpha))_{\ell_1(\Gamma)}.
\]
\end{proposition}
\begin{proof}
The result is immediate as any collection of
$1$-Lipschitz functions $f_\alpha\in\Lip_{N}(N\cup M_\alpha)$
is the restriction of a $\frac1K$-Lipschitz function
$f\in\Lip_{N}(N\cup \cup_{\alpha\in\Gamma}M_\alpha)$
\end{proof}

Recall that the density character $\text{dens}(M)$, or just density,
of a metric space $M$
is the smallest cardinal $\Gamma$ such that there exists dense subset
of $M$ of cardinality $\Gamma$.

Let $\Gamma$ be a cardinal
(which is identified with the smallest ordinal of the same cardinality).
By the cofinality $\text{cof}(\Gamma)$
we denote the smallest ordinal $\alpha$ (in fact a cardinal) such that
$\Gamma=\lim_{\beta<\alpha}\Gamma_\beta$, where $\Gamma_\beta$ is an
increasing ordinal sequence (\cite{J} p.26).

\section{Structural properties}

\begin{proposition}\label{l-1}
Let $M$ be a metric space of density $\text{dens}(M)=\Gamma$. Then $\f(M)$
contains a complemented copy of $\ell_1(\Gamma)$.
\end{proposition}

\begin{proof}
For convenience we may assume that $\Gamma>\omega_0$, because
this case has been already proved in \cite{CDW} (Our proof can be
adjusted to this case as well).
By  (\cite{ro} Corollary 1.2) if $c_0(\Gamma)\hookrightarrow X^*$ then
$\ell_1(\Gamma)$ is complemented in $X$. So it suffices to prove that
$\Lip_0(M)$ contains a copy of $c_0(\Gamma)$.
For every $n\in\N$ let $M_n$ be some maximal $\frac1{2^n}$-separated set in $M$.
Denote $\Gamma_n=|M_n|$. It is clear that
$\text{dens}(M)=\lim_{n\to\infty}\Gamma_n$, in the cardinal sense.
In case when the cofinality $\text{cof}(\Gamma)>\omega_0$, it is clear that
$\Gamma_n=\Gamma$, for some $n\in\N$. In this case, let
$\{f_\alpha:\alpha\in\Gamma_n\}$ be a transfinite sequence
of $1$-Lipschitz functions such that $f_\alpha(x_\alpha)=\frac1{2^{n+3}}$ and
$\text{supp}(f_\alpha)\subset B(x_\alpha,\frac1{2^{n+2}})$. Since
the supports of $f_\alpha$ are pairwise disjoint it is clear that
$\{f_\alpha\}_{\alpha\in\Gamma_n}$ is equivalent to the
unit basis of $c_0(\Gamma)$ and the
result follows. In the remaining case, we may assume that
$\{\Gamma_{k_n}\}_{n=1}^\infty$
is a strictly increasing sequence of cardinals.
Denote $M_n=\{x^n_\alpha\}_{\alpha\in\Gamma_{k_n}}$.
Let $L_1=M_1$. By induction we will construct
sets $L_n\subset M_n$ as follows. Inductive step towards $n+1$.
Consider the sets
\[
A_{j,\alpha}=M_{n+1}\cap B(x_\alpha^j,\frac1{2^{k_j+1}}),\;\;
{j\le n,\alpha\in\Gamma_{k_j}}
\]

If there is some $j,\alpha$ so that $|A_{j,\alpha}|=\Gamma_{k_{n+1}}$
then we let $L_{n+1}=A_{j,\alpha}$. Otherwise we let

\[
L_{n+1}=M_{n+1}\setminus\cup_{j\le n,\alpha\in\Gamma_{k_j}} A_{j,\alpha}
\]

In either case we have $|L_{n+1}|=\Gamma_{k_{n+1}}$.
 By discarding suitable countable
subsets of these sets $L_{n}$ we can assume that
\[
\text{dist}(L_n, L_m)\ge \max\{\frac 1{2^{k_n+1}},\frac 1{2^{k_m+1}}\}
\]

To finish, let
$\{f^n_\alpha: x^n_\alpha\in L_n, n\in\N\}$ be a transfinite sequence
of $1$-Lipschitz disjointly supported
functions such that $f^n_\alpha(x^n_\alpha)=\frac1{2^{k_n+3}}$  and
$\text{supp}(f^n_\alpha)\subset B(x^n_\alpha,\frac1{2^{k_n+2}})$.
 This sequence is equivalent to the basis
 of $c_0(\Gamma)$, which finishes the proof.
\end{proof}

\begin{theorem}\label{eq-n}
Let  $N, M$ be  uniformly discrete infinite
sets of the same cardinality such that $N\subset M$
is a net. Then
$\mathcal F(N)\cong\mathcal F(M)$.
\end{theorem}
\begin{proof}
Let $K>0$ be such that $\max_{m\in M}\dist(m,N)\le K$
Let $r:M\to N$ be a retraction such that $d(x, r(x))\le K$.
As $M$ is uniformly discrete, $r$ is Lipschitz.
By Proposition \ref{ret}

\[
\mathcal F(M)\cong \mathcal F(N)\oplus\mathcal F_N(M).
\]

It is clear that $\mathcal F_N(M)\cong\ell_1(M\setminus N)$.
By Proposition \ref{l-1}

\[
\mathcal F(N)\cong \mathcal F(N)\oplus\ell_1(M)\cong\mathcal F(M).
\]

\end{proof}

Recall that all nets in a given infinite dimensional Banach space are
Lipschitz equivalent (\cite{LMP}, or \cite{benlin} p.239),
hence their free spaces are linearly isomorphic. On the other hand,
there are examples of non-equivalent nets in $\R^2$ (\cite{Mc}, \cite{BuK}
or \cite{benlin} p.242), hence the next result is not immediately obvious.

\begin{proposition}\label{sam-f}
Let  $\mathcal N$, $\m$ be nets of the same cardinality $\text{dens}(M)$
in a metric space  $(M,d)$. Then
$\mathcal F(\mathcal N)\cong\mathcal F(\m)$.
\end{proposition}
\begin{proof}
Suppose $\mathcal N$ is a $(a,b)$-net and $\m$ is a $(c,d)$-net in $M$, $a\le c$.
Let $K=\m\cup\mathcal N$, and let $\mathcal K\subset K$ be maximal
subset such that from each pair of points $x\in\m, y\in\mathcal N$
for which $d(x,y)<\frac a4$ we choose only one $x\in\mathcal K$.
It is now clear that both $\mathcal N$ and $\m$ are bi-Lipschitz
equivalent to a respective subset of $\mathcal K$. By Theorem \ref{eq-n},
$\f(\mathcal K)\cong\f(\m)\cong\mathcal F(\mathcal N)$.
\end{proof}

Of course, the above proposition applies to any pair of nets in a given
Banach space $X$,
or its subset $S\subset X$ which contains arbitrarily large balls.

\begin{lemma}\label{plus}
Let $Y=X\oplus\R$ be Banach spaces, $\mathcal N$ be a net in $X$
and $\m$ be the extension of $\mathcal N$ into the natural net
in $Y$. Denote $\m^+=\m\cap X\oplus\R^+$,  $\m^-=\m\cap X\oplus\R^-$.

If $\mathcal F(\mathcal N)=\f(\mathcal N)\oplus\f(\mathcal N)$ and
$\f(\m^+)=\f(\m^+)\oplus\f(\m^+)$
then $\f(\m)=\f(\m^+)=\f(\m)\oplus\f(\m)$.
\end{lemma}

\begin{proof}

Thanks to Proposition \ref{sam-f} we are allowed to make additional assumptions
on the form of the nets.
Let us assume that $\m=\mathcal N\times\Z$, which immediately implies that
$\mathcal N\cup\m^+$ is bi-Lipschitz equivalent with $\m^+$ (and $\m^-$) by translation.
Denoting $P:Y\to X$ the canonical projection
$P(x,t)=x$, we see that $P:\m\to\mathcal N$ is a Lipschitz retraction, so
\[
\f(\m^+)\cong\f(\mathcal N\cup\m^+)\cong
\f(\mathcal N)\oplus\f_{\mathcal N}(\mathcal N\cup\m^+)
\]
and using Proposition \ref{l-1-sum}
\[
\f(\m)\cong\f(\mathcal N)\oplus\f_{\mathcal N}(\m)\cong
\f(\mathcal N)\oplus\f_{\mathcal N}(\mathcal N\cup\m^+)
\oplus\f_{\mathcal N}(\mathcal N\cup\m^-)
\]

Since $\f_{\mathcal N}(\mathcal N\cup\m^+)
\cong\f_{\mathcal N}(\mathcal N\cup\m^-)$
and
$\f({\mathcal N})\cong\f({\mathcal N})\oplus\f({\mathcal N})$ the result
follows.

\end{proof}

\begin{theorem}
Let ${\mathcal N}$ be a net in $\R^n$. Then
$\f({\mathcal N})\cong\f({\mathcal N})\oplus\f({\mathcal N})$.
\end{theorem}

\begin{proof}

For $n=1$ it is well known \cite{Gd} that
$\f({\mathcal N})\cong\f({\mathcal N}^+)\cong\ell_1\cong
 \f(\mathcal N)\oplus\f(\mathcal N)$.

Inductive step towards $n+1$. We may assume that $\mathcal N=\Z^{n+1}$
is the integer grid. Let us use the following notation (our convention
is that $\Z^+=\{1,2,3,\dots\}, \Z^-=\{-1,-2,\dots\}$).
\[
\mathcal L=\Z^{n-1}\times\{0\}\times\{0\},\;
\mathcal L_1=\Z^{n-1}\times\Z^+\times\{0\},\;
\mathcal L_2=\Z^{n-1}\times\{0\}\times\Z^+,\;
\mathcal L_3=\Z^{n-1}\times\Z^-\times\{0\}
\]

\[
\mathcal M^+=\Z^{n-1}\times\Z\times\Z^+,\;
\mathcal M_1=\Z^{n-1}\times\Z^+\times\Z^+,\;
\mathcal M_2=\Z^{n-1}\times\Z^-\times\Z^+
\]

With this notation, we have the following bi-Lipschitz equivalence

\[
\mathcal L_1\cup\mathcal L\cup\mathcal L_2\cong
\mathcal L_1\cup\mathcal L\cup\mathcal L_3.
\]

By inductive assumption this implies

\begin{equation}\label{ilo}
\mathcal F(\mathcal L\cup\mathcal L_1\cup\mathcal L_2)
\cong
\mathcal F(\mathcal L\cup\mathcal L_1\cup\mathcal L_2)
\oplus
\mathcal F(\mathcal L\cup\mathcal L_1\cup\mathcal L_2).
\end{equation}

On the other hand, using Proposition \ref{l-1-sum} in various settings

\[
\mathcal F(\mathcal L\cup\mathcal L_1\cup\mathcal L_2)
\cong
\mathcal F(\mathcal L)\oplus
\mathcal F_{\mathcal L}(\mathcal L\cup\mathcal L_1)\oplus
\mathcal F_{\mathcal L}(\mathcal L\cup\mathcal L_2),
\]

\begin{equation}\label{hve}
\mathcal F(\mathcal L\cup\mathcal L_1\cup\mathcal L_2)
\cong
\mathcal F(\mathcal L\cup\mathcal L_1)
\oplus \mathcal F_{\mathcal L\cup\mathcal L_1}
(\mathcal L\cup\mathcal L_1\cup\mathcal L_2)\cong
\mathcal F(\mathcal L\cup\mathcal L_1)
\oplus \mathcal F_{\mathcal L}
(\mathcal L\cup\mathcal L_2),
\end{equation}

\[
\mathcal F(\mathcal L\cup\mathcal L_1\cup\mathcal L_2\cup\mathcal L_3)
\cong
\mathcal F(\mathcal L\cup\mathcal L_1\cup\mathcal L_3)
\oplus \mathcal F_{\mathcal L\cup\mathcal L_1\cup\mathcal L_3}
(\mathcal L\cup\mathcal L_1\cup\mathcal L_2\cup\mathcal L_3)\cong
\mathcal F(\mathcal L\cup\mathcal L_1\cup\mathcal L_3)
\oplus \mathcal F_{\mathcal L}
(\mathcal L\cup\mathcal L_2).
\]

Hence, using the inductive assumption
$\mathcal F(\mathcal L\cup\mathcal L_1\cup\mathcal L_3)\cong
\mathcal F(\mathcal L\cup\mathcal L_1)$

\begin{equation}\label{kri}
\mathcal F(\mathcal L\cup\mathcal L_1\cup\mathcal L_2\cup\mathcal L_3)
\cong
\mathcal F(\mathcal L\cup\mathcal L_1)
\oplus \mathcal F_{\mathcal L}
(\mathcal L\cup\mathcal L_2)
\end{equation}

Comparing \eqref{hve}, \eqref{kri} and using \eqref{ilo} we obtain

\begin{equation}\label{l-ji}
\mathcal F(\mathcal L\cup\mathcal L_1\cup\mathcal L_2\cup\mathcal L_3)
\cong
\mathcal F(\mathcal L\cup\mathcal L_1\cup\mathcal L_2)
\cong
\mathcal F(\mathcal L\cup\mathcal L_1\cup\mathcal L_2)
\oplus
\mathcal F(\mathcal L\cup\mathcal L_1\cup\mathcal L_2)
\end{equation}

By Lemma \ref{plus},
in order to complete the inductive step, it suffices to prove that
$\f(\m^+)=\f(\m^+)\oplus\f(\m^+)$.

Denote $R:\R^+\times\R^+\to\R\times\R^+$ the mapping $R(z)=\frac{z^2}{|z|}$,
where $z$ is the complex number represented as $z=x+iy$. It is clear that
$R$ is bi-Lipschitz. Indeed, if $z_0=a+ib$ and $z_1=x+iy$ are two complex numbers from the first quadrant with $a\leq x$, then
\begin{align*}
|R(z_0)-R(z_1)|&=|e^{a+2ib}-e^{x+2iy}|\leq |e^{a+2ib}-e^{a+2iy}|+|e^{a+2iy}-e^{x+2iy}|=\\
&=e^a|e^{ib}-e^{iy}|\cdot|e^{ib}+e^{iy}|+|e^a-e^x|\leq 2|e^{a+ib}-e^{a+iy}| + |e^{a+ib}-e^{x+iy}|\\
&\leq 2|e^{a+ib}-e^{x+iy}|+|e^{a+ib}-e^{x+iy}|=3|z_0-z_1|.
\end{align*}
On the other hand, for any $z_0=a+ib$ and $z_1=x+iy$ from the upper half plane with $a\leq x$ we have
\begin{align*}
|R^{-1}(z_0)-R^{-1}(z_1)|&=|e^{a+\frac{ib}{2}}-e^{x+\frac{iy}{2}}|\leq |e^{a+\frac{ib}{2}}-e^{a+\frac{iy}{2}}|+|e^{a+\frac{iy}{2}}-e^{x+\frac{iy}{2}}|=\\
&=e^a\frac{|e^{ib}-e^{iy}|}{|e^{\frac{ib}{2}}+e^{\frac{iy}{2}}|}+|e^a-e^x|\\
&\leq \frac{\sqrt{2}}{2}|e^{a+ib}-e^{a+iy}| + |e^{a+ib}-e^{x+iy}|\\
&\leq 2|z_0-z_1|,
\end{align*}
which we wanted to prove.

The mapping
\[
T:\m_1\to \R^{n+1},
T(u,x,y)=(u,R(x,y))
\]
takes the net $\m_1$ from the set $\R^{n-1}\times\R^+\times\R^+$
in a bi-Lipschitz way
to the net $T(\m_1)$ in the set $\R^{n-1}\times\R\times\R^+$.
Hence $\f(\m_1)\cong\f(T(\m_1))$. Since $\m^+=\m_1\cup\mathcal L_2\cup\m_2$
is another net in the second set, by Proposition \ref{sam-f} we obtain

\[
\f(\m_1)\cong\f(\m^+)
\]

Now thanks to the bi-Lipschitz equivalence
$\m_1\cong\m_1\cup\mathcal L\cup\mathcal L_1\cup\mathcal L_2$,

\[
\f(\m_1)\cong\f(\m_1\cup\mathcal L\cup\mathcal L_1\cup\mathcal L_2)\cong
\f(\mathcal L\cup\mathcal L_1\cup\mathcal L_2)\oplus
\f_{\mathcal L\cup\mathcal L_1\cup\mathcal L_2}
(\m_1\cup\mathcal L\cup\mathcal L_1\cup\mathcal L_2)
\]

Since $\m^+$ is bi-Lipschitz equivalent to
$\m^+\cup\mathcal L\cup\mathcal L_1\cup\mathcal L_2$ we get

\[
\f(\m^+)\cong
\f(\mathcal L\cup\mathcal L_1\cup\mathcal L_2\cup\mathcal L_3)\oplus
\f_{\mathcal L\cup\mathcal L_1\cup\mathcal L_2}
(\m_1\cup\mathcal L\cup\mathcal L_1\cup\mathcal L_2)
\oplus
\f_{\mathcal L\cup\mathcal L_2\cup\mathcal L_3}
(\m_2\cup\mathcal L\cup\mathcal L_1\cup\mathcal L_3)
\]

Using \eqref{l-ji} and the obvious
$\f_{\mathcal L\cup\mathcal L_1\cup\mathcal L_2}
(\m_1\cup\mathcal L\cup\mathcal L_1\cup\mathcal L_2)
\cong
\f_{\mathcal L\cup\mathcal L_2\cup\mathcal L_3}
(\m_2\cup\mathcal L\cup\mathcal L_1\cup\mathcal L_3)$

we finally obtain

\[
\f(\m_1)\oplus\f(\m_1)\cong\f(\m^+)\cong\f(\m_1)
\]

which ends the inductive step and the proof.
\end{proof}

\begin{theorem}\label{pol-l}
Let $X$ be a Banach space such that
$X\cong Y\oplus X$, where
$Y$ is an infinite dimensional Banach space with a Schauder  basis.
Let ${\mathcal N}$ be a net in $X$. Then
\[
\mathcal F({\mathcal N})\cong(\oplus_{j=1}^\infty F({\mathcal N}))_{\ell_1}.
\]
\end{theorem}

\begin{proof}
We may assume without loss of generality that the norm of the direct sum $Y\oplus X$
is in fact equal to the maximum norm $Y\oplus_\infty X$.
Using Proposition \ref{sam-f} it suffices to
prove the result for just one particular net $\mathcal N$.
Let $M_k\subset k S_X, k\in\N$ be a $(1,2)$-net. Then
${\mathcal N}=\cup_{k=1}^\infty  M_k$ is a $(1,3)$-net
in $X$. Let $\{e_k\}$ be a bi-monotone normalized Schauder basis of $Y$.
Set $Z=(\oplus_{j=1}^\infty F({\mathcal N}))_{\ell_1}$.
It is clear that $Z\cong(\oplus_{j=1}^\infty Z)_{\ell_1}$

We will use Pelczynski's decomposition technique to
prove the theorem.
Since $\f(\mathcal N)$ is complemented in $Z$
it only remains to prove that $\f(\mathcal N)$ contains a complemented
subspace isomorphic to $Z$. Let

\[
V_n=\{k e_n\oplus x: x\in M_k, k\in\N\}\subset
Y\oplus X
\]
\[
M=\cup_{n=1}^\infty V_n
\]

The sets $V_n$, as subsets of the pointed metric space
$(Y\oplus X,\|\cdot\|,0)$, satisfy the assumptions of
Proposition \ref{l-1-sum} and so
\[
\f(M)=(\oplus_{n=1}^\infty\f(V_n))_{\ell_1}\cong
(\oplus_{n=1}^\infty\f(\mathcal N))_{\ell_1}
\cong Z.
\]

We extend the set $M$ into a $(1,3)$-net $\m$ in $Y\oplus X$.
Because $\f(\m)\cong\f(\mathcal N)$
it suffices to show that $\f(\m)$ contains a complemented copy of $\f(M)$.
To this end it is enough to
 find a Lipschitz retraction $R:\m\to M$. Denote by $[a]$ the integer 
 part of $a\in\R$. First let
$r:X\to{\mathcal N}$ be a
(non-continuous) retraction such that $[\|x\|]\le\|r(x)\|\le\|x\|,
\|r(x)-x\|\le4$ and $\|r(x)\|=\|x\|$
provided $\|x\|\in\N$. Let
$s:Y\to Y$ be a (non-continuous) retraction defined
for $x=\sum_{i=1}^\infty x_i e_i$ by

\begin{equation}
 s\left(\sum_{i=1}^\infty x_i e_i\right)=\begin{cases}
 d e_k&\text{ if } x_k>\max\{x_i: i\ne k\}\cup\{0\}, d=\min_{i\ne k}[x_k-x_i]\\
 \\  0& \text{otherwise }\end{cases}
\end{equation}

It is easy to see that $\|r(x)-r(y)\|\le 9\|x-y\|$,
$\|s(x)-s(y)\|\le 6\|x-y\|$ provided $\|x-y\|\ge1$ (i.e. they are
Lipschitz for large distances). Indeed,
\[
\|r(x)-r(y)\|\le\|r(x)-x\|+\|r(y)-y\|+\|x-y\|\le8+\|x-y\|\le9\|x-y\|
\]
Assuming $1\le\|x-y\|\le\lambda$, we get $|x_i-y_i|\le\lambda, i\in\N$. Suppose
that $s(x)=de_k$, $s(y)=te_l$. We claim that $d\le3\lambda$.
Indeed, assuming by contradiction that
$x_k\ge d+\max\{x_i: i\ne k\}\ge 3\lambda+\max\{x_i: i\ne k\}$
we obtain that $y_k\ge\lambda+\max\{y_i: i\ne k\}$. Hence $k=l$
and $|d-t|\le2\lambda+2$. The same argument yields $t\le3\lambda$,
so finally we obtain $\|s(x)-s(y)\|\le6\lambda$.

Let $R:\m\to M$ is now defined as

\begin{equation}
R(y\oplus x)=\begin{cases}
s(y)\oplus r\left(\frac{\|s(y)\|}{\|x\|} x\right)&\text{ if }
\|x\|>\|s(y)\|>0\\
\frac{\|r(x)\|}{\|s(y)\|}s(y)\oplus r(x)&\text{ if }
\|s(y)\|\ge\|x\|>0\\0& \text{otherwise }\end{cases}
\end{equation}

We claim that $R$ is a retraction onto $M$. If $y\oplus x\in M$ then clearly
$s(y)=y, r(x)=x$, $\|s(y)\|=\|r(x)\|$ and so $R(y\oplus x)=y\oplus x$. Next,
observe that $R(y\oplus x)\in M$ holds for every $y\oplus x\in\mathcal M$.
Indeed, regardless of the case in the definition of $R$, we see that
the first summand of $R(y\oplus x)$
is a non-negative integer multiple of some basis vector 
$e_n$ in $Y$.
In the first (and third) case it is obvious, in the second case it follows
as the norm of $\frac{\|r(x)\|}{\|s(y)\|}s(y)$ is an integer $\|r(x)\|$.
The second summand is the result of an application of the retraction
$r$, and its norm
equals the norm of the first summand, hence the value of $R(y\oplus x)$
indeed lies in $M$. 

Next, we claim that $R$ is Lipschitz.
Recall that $\mathcal M$ is a $(1,3)$-net in a Banach
space, so it suffices
to prove that there exists a $K>0$ such that $\|R(y_1\oplus x_1)-R(y\oplus x)\|\le K$
whenever $\|y_1\oplus x_1-y\oplus x\|\le D$, for say  $D=8$. This is well-known
and easy to see, as every pair of distinct points  $p,q\in\mathcal M$ can be 
connected by a straight  segment of length $\|p-q\|$, and a sequence
of $[\frac1{\|p-q\|}]+1$ points on this segment of distance (of consecutive
elements) at most one. Each
of these points has an approximant from $\mathcal M$ of distance at most $3$,
so it clear that there exists a sequence of $[\frac1{\|p-q\|}]+1$ points
in $\mathcal M$ of (consecutive) distance at most $D-1=7$, "connecting"
the points $p,q$, and the result follows by a simple summation of the
increments of $R$ along the mentioned sequence.

Let us start the proof of Lipschitzness of $R$ by partitioning $\mathcal M$ 
into three
disjoint subsets 

\[
D_1=\{y\oplus x: \|x\|>\|s(y)\|\ge 20D\},
\]
\[
D_2=\{y\oplus x: \|s(y)\|\ge\|x\|\ge 20D\},
\]
\[
D_3=\{y\oplus x: \min\{\|s(y)\|,\|x\|\}< 20D\}.
\]

The set $D_1$ (resp. $D_2$) corresponds to the case $1$ (resp. $2$)
in the definition of $R$.

Observe that $\|R(y\oplus x)\|\le\min\{\|y\|,\|x\|\}$ so it suffices to
prove the Lipschitzness of $R$ on the set $D_1\cup D_2$.
Moreover, the sets $D_1$ and $D_2$ have in a sense a common 
"boundary" (in the intuitive sense,
which is not contained in $D_1$) consisting of those elements
for which $\|x\|=\|s(y)\|$. It is easy to see that for such elements
the first two cases in definition of $R$ may be applied with the same result
(although formally we are forced to apply the second case). 
Suppose now that $p\in D_1, q\in D_2$. A similar argument as above
using the straight segment connecting $p,q$ (and a short finite sequence
from $\mathcal M$ which approximates this segment) we see that the segment
essentially has to "cross the boundary" between $D_1, D_2$, and so the
proof of the Lipschitzness of $R$ will follow provided
we can do it for each of the sets $D_1, D_2$ separately.

Suppose $y_1=y+\tilde y, x_1=x+\tilde x$ are such that 
$\|\tilde y\|,\|\tilde x\|\le D$. 

\vvs 

Case 1.
We consider first the case $y_1\oplus x_1, y\oplus x\in D_1$. Then

\[
\frac{\|s(y_1)\|}{\|x_1\|}x_1-\frac{\|s(y)\|}{\|x\|}x=
\frac{\|s(y+\tilde y)\|}{\|x+\tilde x\|}(x+\tilde x)-\frac{\|s(y)\|}{\|x\|}x
=
\bo\frac{\|s(y+\tilde y)\|}{\|x+\tilde x\|}-\frac{\|s(y)\|}{\|x\|}\bc x
+\frac{\|s(y+\tilde y)\|}{\|x+\tilde x\|}\tilde x
\]

Now

\[
\left|\frac{\|s(y+\tilde y)\|}{\|x+\tilde x\|}-\frac{\|s(y)\|}{\|x\|}\right|
\le 
\max\{\frac{\|s(y)\|+9D}{\|x\|-D}-\frac{\|s(y)\|}{\|x\|} ,
\frac{\|s(y)\|}{\|x\|}-\frac{\|s(y)\|-9D}{\|x\|+D}\}
\]

\[
\frac{\|s(y)\|+9D}{\|x\|-D}-\frac{\|s(y)\|}{\|x\|}=
\frac{(\|s(y)\|+9D)\|x\|-\|s(y)\|(\|x\|-D)}{(\|x\|-D)\|x\|}
\]
\[
=\frac{9D\|x\|+D\|s(y)\|}{(\|x\|-D)\|x\|}\le
\frac{10D\|x\|}{(\|x\|-D)\|x\|}\le\frac{10D}{\frac{9}{10}\|x\|}\le\frac{100D}{9\|x\|}\leq \frac{12D}{\|x\|}
\]

Similarly, we obtain

\[
\frac{\|s(y)\|}{\|x\|}-\frac{\|s(y)\|-9D}{\|x\|+D}\le\frac{10D}{\|x\|}.
\]

Hence we obtain

\[
\left|\frac{\|s(y+\tilde y)\|}{\|x+\tilde x\|}-\frac{\|s(y)\|}{\|x\|}\right|
\le\frac{12D}{\|x\|}.
\]

The last term is also estimated similarly:
\[
\frac{\|s(y+\tilde y)\|}{\|x+\tilde x\|}\|\tilde x\|
\le\frac{\|s(y)\|+9D}{\|x\|-D}D\le
\frac{\|x\|+9D}{\|x\|-D}D\le3D
\]

So combining the above computations we get

\[
\left\|\frac{\|s(y_1)\|}{\|x_1\|}x_1-\frac{\|s(y)\|}{\|x\|}x\right\|\le
15D
\]

So the mapping $y\oplus x\to\frac{\|s(y)\|}{\|x\|}x$ from $D_1$ to $M$ takes 
vectors of distance at most $D$ to vectors of distance at most $15D$. 
It is now clear that $R$ is Lipschitz on $D_1$.

\vvs

Case 2.
We consider now $y_1\oplus x_1, y\oplus x\in D_2$, and denote $z=s(y+\tilde y)-s(y)$
(recall that $\|z\|\le9D$):

\[
\frac{\|r(x_1)\|}{\|s(y_1)\|}s(y_1)-
\frac{\|r(x)\|}{\|s(y)\|}s(y)=
\frac{\|r(x+\tilde x)\|}{\|s(y+\tilde y)\|}s(y+\tilde y)-
\frac{\|r(x)\|}{\|s(y)\|}s(y)
\]
Therefore
\[
\left\|\frac{\|r(x+\tilde x)\|}{\|s(y+\tilde y)\|}s(y+\tilde y)-
\frac{\|r(x)\|}{\|s(y)\|}s(y)\right\|\]
\[
\leq\max\{\left\|\frac{\|r(x)\|+9D}{\|s(y)\|-9D}(s(y)+z)-
\frac{\|r(x)\|}{\|s(y)\|}s(y)\right\|,\left\|\frac{\|r(x)\|-9D}{\|s(y)\|+9D}(s(y)+z)-
\frac{\|r(x)\|}{\|s(y)\|}s(y)\right\|\}
\]

The first term could be rewritten and estimated as follows:

\[
\left\|\frac{(\|r(x)\|+9D)\|s(y)\|}{(\|s(y)\|-9D)\|s(y)\|}(s(y)+z)-
\frac{(\|s(y)\|-9D)\|r(x)\|}{(\|s(y)\|-9D)\|s(y)\|}s(y)\right\|
\]
 
\[
\le\left\|\bo\frac{(\|r(x)\|+9D)\|s(y)\|}{(\|s(y)\|-9D)\|s(y)\|}-
\frac{(\|s(y)\|-9D)\|r(x)\|}{(\|s(y)\|-9D)\|s(y)\|}\bc s(y)\right\|+
\frac{\|r(x)\|+9D}{\|s(y)\|-9D}9D
\] 
 
\[
\le\left\|\bo\frac{(\|r(x)\|+9D)\|s(y)\|-
(\|s(y)\|-9D)\|r(x)\|}{(\|s(y)\|-9D)\|s(y)\|}\bc s(y)\right\|+27D
\]
 
\[
\le\left|\frac{(\|r(x)\|+9D)\|s(y)\|-
(\|s(y)\|-9D)\|r(x)\|}{\|s(y)\|-9D}\right|+
27D
\] 

\[
\le\left|\frac{9D\|s(y)\|+9D\|r(x)\|}{\|s(y)\|-9D}\right|+
27D\le\frac{18D\|s(y)\|}{\|s(y)\|-9D}+
27D\le 63D.
\]   

The second term we estimate analogously

$$\left\|\frac{\|r(x)\|-9D}{\|s(y)\|+9D}(s(y)+z)-
\frac{\|r(x)\|}{\|s(y)\|}s(y)\right\|$$
\[
\le\left\|\bo\frac{(\|r(x)\|-9D)\|s(y)\|}{(\|s(y)\|+9D)\|s(y)\|}-
\frac{(\|s(y)\|+9D)\|r(x)\|}{(\|s(y)\|+9D)\|s(y)\|}\bc s(y)\right\|+
\frac{\|r(x)\|-9D}{\|s(y)\|+9D}9D
\] 
\[
\le\left|\frac{9D\|s(y)\|+
9D\|r(x)\|}{\|s(y)\|+9D}\right|+
9D\leq \frac{18D\n s(y)\n}{\n s(y)\n}+9D\leq 27D.
\] 

 We conclude that $R$ is Lipschitz on the whole domain $\mathcal M$.  
 Hence $\mathcal F(M)$ is
isomorphic to a complemented subspace of $\f(\m)\cong\f({\mathcal N})$.
\end{proof}

A simple situation which fits the above assumptions is when
$X$ contains a complemented subspace with a symmetric basis
(e.g. $\ell_p$, $c_0$ or an Orlicz sequence space).
By the standard structural theorems for classical Banach spaces (\cite{LT})
we obtain.

\begin{corollary}
Let $X$ be a Banach space isomorphic to any of the (classical) spaces
$\ell_p, L_p, 1\le p<\infty, C(K)$, or an Orlicz space $h_M$,
$\mathcal N$ be a net in $X$. Then
\[
\mathcal F({\mathcal N})\cong(\oplus_{j=1}^\infty F({\mathcal N}))_{\ell_1}.
\]
\end{corollary}

Recall that a metric space $M$ is an absolute
 Lipschitz retract if, for some $K>0$, $M$ is a $K$-Lipschitz retract of every
metric superspace $M\subset N$ (\cite{benlin} p.13).
We are going to use the discretized form of this condition. This concept
is almost explicit in the work of Kalton \cite{K3}, where it would
probably be called absolute coarse retract.

\vs

\begin{definition}
Let $M$ be a $\delta$-uniformly discrete space, $\delta>0$.
We say that $M$ is an absolute
uniformly discrete Lipschitz retract if, for some $K>0$, the space $M$ is a $K$-Lipschitz retract of every
$\delta$-uniformly discrete superspace $M\subset N$.
\end{definition}

\begin{lemma}
Let $X$ be Banach space which is an absolute Lipschitz retract, ${\mathcal N}$ be a net
in $X$. Then ${\mathcal N}$ is absolute
uniformly discrete Lipschitz retract. Conversely,
if ${\mathcal N}$ is absolute
uniformly discrete Lipschitz retract and $X$ is a Lipschitz retract of $X^{**}$
then $X$ is an absolute Lipschitz retract.
\end{lemma}

\begin{proof}
The first implication is obvious. To prove the second one,
suppose that $X\subset\ell_\infty(\Gamma)=Y$ is a linear embedding.
Since $\ell_\infty(\Gamma)$ is an injective space,
it suffices to prove that there is a Lipschitz retraction from
$\ell_{\infty}(\Gamma)$ onto $X$. Since $X$ is a Lipschitz retract of
$X^{**}$, it suffices to follow verbatim the proof of  Theorem 1 in \cite{lin}.
Indeed,
consider a net $\mathcal N$ in $X$
with extension into a net
$\m$ in $Y$. By assumption, there exists a Lipschitz retraction
$r:\m\to\mathcal N$. This retraction $r$ can be easily extended to a coarsely
continuous retraction $R$ from $Y$ onto $X$
(using the terminology of \cite{K3}), which is of course
Lipschitz for large distances. It is this condition on $R$ that is
used in the proof of Theorem 1 in \cite{lin}.

\end{proof}

Remark. It is an open problem if the retraction from $X^{**}$ to $X$ exists
for every separable Banach space (see \cite{K3}).

Important examples of absolute uniformly discrete Lipschitz retract
are the nets in  $C(K)$ spaces, $K$ metrizable compact,
\cite{benlin} p.15.

\vs

\begin{corollary}
Let $\m$ be a countable absolute
uniformly discrete Lipschitz retract which contains a bi-Lipschitz
copy of the net $\mathcal N$ in $c_0$.
Then $\f(\m)\cong\mathcal F({\mathcal N})$.
\end{corollary}
\begin{proof}
There is a Lipschitz retraction from $\m$ onto
$\mathcal N$, and on the other hand
using Aharoni's theorem (\cite{FHHMZ} p. 546) $\m$ is bi-Lipschitz
embedded into $\mathcal N$ (and hence also a retract).
Thus
$\f(\m)$ is complemented in $\mathcal F({\mathcal N})$
and vice versa. To finish, apply Theorem \ref{pol-l} for $c_0$
together
with the Pelczynski decomposition principle.
\end{proof}

To give concrete applications of the above corollary, we obtain
the following result. The case of $c^+_0$ follows from the Pelant
$c^+_0$-version of Aharoni's result \cite{Pe}.

\begin{theorem}\label{ck}
Let  $\mathcal N$ be a net in $c_0$ and $\m$ be a net in any of the
following metric spaces: $C(K)$, $K$ infinite metrizable compact, or
$c_0^+$ (the subset of $c_0$ consisting of elements
with non-negative coordinates).
Then $\f(\m)\cong\mathcal F({\mathcal N})$.
\end{theorem}

\section{Schauder basis}

\begin{theorem}\label{one} Let $X$ be a metric space. Suppose there exist a set $M\sub X$ and a sequence of distinct points $\{\mu_n\}_{n=1}^{\infty}\subseteq M$, together with a sequence of retractions
$\{\fii_n\}_{n=1}^\infty$, $\fii_n:M\to M$, $n\in\N$, which satisfy the following conditions:
\begin{enumerate}[(i)]
\item $\fii_n(M)=M_n:=\bigcup_{j=1}^{n}\{\mu_j\}$ for every $n\in\N$,\label{bed1}
\item $\overline{\bigcup_{j=1}^{\infty}\{\mu_j\}}^X=M$, \label{bed2}
\item There exists $K>0$ such that $\fii_n$ is $K$-Lipschitz for every $n\in\N$,\label{bed3}
\item $\fii_m\fii_n=\fii_n\fii_m=\fii_n$ for every $m,n\in\N$, $n\leq m$. \label{bed4}
\end{enumerate}
Then the space $\F(M)$ has a Schauder basis with the basis constant at most $K$.
\end{theorem}
\begin{proof}
It is a well-known fact that every Lipschitz mapping $L:A\to B$ between pointed metric spaces $A,B$, such that $L(0)=0$ extends uniquely to a linear mapping\\ $\widehat{L}:\F(A)\to\F(B)$ in a way that that the following diagram commutes:
$$\begin{CD}
\F(A) @>{\widehat{L}}>> \F(B)\\
@A{\d_A}AA @AA{\d_B}A\\
A @>{L}>> B
\end{CD}$$
Moreover, the norm of $\widehat{L}$ is at most $\lip(L)$. Therefore for every $n\in\N$ there is a  linear mapping $P_n=\widehat{\fii_{n}}:\F(M)\to\F(M)$ extending $\fii_n:M\to M$ with $\n P_n\n\leq K$. We want to prove that $\{P_n\}$ is a sequence of canonical projections associted with some Schauder basis of $\F(M)$, namely that

\begin{enumerate}[a)]
\item $\dim P_n(\F(M))=n-1$ for every $n\in\N$,
\item $P_nP_m=P_mP_n=P_m$ for all $m,n\in\N$, $m\leq n$,
\item $\lim_{n}P_n(x)=x$ for all $x\in\F(M)$.
\end{enumerate}

The first condition is easy: as $\fii_n(M)=M_n=\{\mu_i\}_{i=1}^{n}$ we have $P_n(\F(M))=\F(M_n)$, which is a $(n-1)$-dimensional space. Let us check the commutativity. Note first that for $m,n\in\N$ the diagram
$$\begin{CD}
\F(M) @>{P_m}>> \F(M)@>{P_n}>> \F(M)\\
@A{\d_M}AA @AA{\d_M}A @AA{\d_M}A\\
M @>{\fii_m}>> M @>{\fii_n}>> M
\end{CD}$$
commutes, which means that $\widehat{\fii_n\circ\fii_m}=P_nP_m$. But then from the condition \ref{bed4} follows $P_nP_m=P_mP_n=P_m$ for $m\leq n$.

The validity of the limit equation is proved easily. Note that elements of the form $\sum_{i=1}^m\a_i\delta_{x_i}$, where $m\in\N$, $x_i\in \{\mu_n\}_{n=1}^\infty$, $\a_i\in\R$ for all $i\in\{1,\cdots,m\}$, are norm dense in $\F(M)$. Indeed, it is a well-known fact that elements $\mu\in\F(M)$ of the same form $\sum_{i=1}^m\a_i\delta_{x_i}$ with $x_i\in M$ are norm dense in $\F(M)$ and the condition \ref{bed2} gives the more general result. By uniform boundedness of the family $\{P_n\}_{n=1}^{\infty}$, it suffices to check the limit for elements mentioned above. Thus pick a measure $\mu=\sum_{i=1}^m\a_i\delta_{x_i}$, $m\in\N$, $a_i\in\R$, $x_i\in\{\mu_j\}_{j=1}^{\infty}$ for all $i\in\{1,...,m\}$. Find $k\in\N$ such that $\{x_1,\cdots,x_m\}\subseteq M_k$. Then for all $n\geq k$ we have
\begin{align*}
\n P_n\mu-\mu\n&=\sup_{\n f\n\leq 1}\left|\langle f,\sum_{i=1}^m\a_i(\delta_{\fii_n(x_i)}-\delta_{x_i})\rangle\right|=\sup_{\n f\n\leq 1}\left|\sum_{i=1}^m\left(\a_i f(\fii_n(x_i))-\a_if(x_i)\right)\right|=\\
&=\sup_{\n f\n\leq 1}\left|\sum_{i=1}^m\left(\a_i f(x_i)-\a_if(x_i)\right)\right|=0.
\end{align*}
This was to prove.
\end{proof}
\begin{definition} Let $X$ be a Banach space with a Schauder basis $E=\{e_i\}_{i=1}^\infty$. The set $M(E)=\{x\in X|\ x=\sum_{i=1}^{\infty}x_ie_i,\ x_i\in\Z,i\in\N\}$ we call the integer-grid to the basis $E$. If it is clear what basis we are working with, we will denote the set $M$ and speak simply about a grid.
\end{definition}
It is not difficult to see that if a basis $E$ is normalized, then the grid $M(E)$ is a $\frac{1}{2\text{bc}(E)}$-separated set, where $\text{bc}(E)$ denotes the basis constant of $E$. For $E$ an unconditional basis we will denote $\text{uc}(E)$ the unconditional constant of $E$. We will now show that for a normalized, unconditional basis $E$ the space $\F(M)$ has a Schauder basis.
\begin{lemma}\label{dva} Let $X$ be a Banach space with a normalized, unconditional Schauder basis $E=\{e_i\}_{i\in\N}$ and a grid $M(E)=M$. Then there exists a sequence of retractions $\fii_n:M\to M$ together with a sequence of distinct points $\mu_n\in M$, $n\in\N$ satisfying the conditions from the Theorem \ref{one} with the constant at most $K=\text{uc}(E)+2\text{bc}(E)$.
\end{lemma}
\begin{proof}
Before we define the retractions $\{\fii_n\}_{n=1}^\infty$ and the points $\{\mu_n\}_{n=1}^\infty$ rigorously, let us give the reader some geometric idea of how will the retractions look like. We will add points from $M$ so that first the set $C_1^1=\{x_1e_1|\ |x_1|\leq 1\}$ is created, then the set $C_1^2=\{x_1e_1+x_2e_2|\ |x_i|\leq 1,i=1,2\}$, then the set $C_2^2=\{x_1e_1+x_2e_2|\ |x_i|\leq 2,i=1,2\}$, then $C^3_2=\{\sum_{i=1}^3x_ie_i|\ |x_i|\leq 2,i=1,2,3\}$ and so on. Note that coordinates of each $\mu\in C_i^j$ are entire numbers.

The retractions will cut coordinates of the argument so that if $x=\sum_{i=1}^\infty x_ie_i\in M$ and $\{\mu_i\}_{i=1}^n=M_n=\fii_n(M)$, $n\in\N$, then $\fii_n(x)$ is obtained by following algorithm: Choose all $\mu_i\in M_n$ minimizing the value $|x_1-(\mu_i)_1|$, out of them choose those $\mu_{i_j}$ minimizing $|x_2-(\mu_{i_j})_2|$ and so on. Note the process will stop eventually because $x=\sum_{i=1}^kx_ie_i$ for some $k\in\N$ as $x\in M$ and the basis $E$ is normalized. It will be a matter of choosing (ordering) the points $\{\mu_i\}_{i=1}^\infty$ so that the process ends with only one point $\mu_i=\fii_n(x)$.
\\
\\We are now going to describe the construction of the sequence $\fii_n$ in the following way. We will build the sequence of points $\mu_n$ and to each $n\in\N$, we associate the sets $\fii_n^{-1}(\mu_i)$, $i\in\{1,\cdots,n\}$. As we want the image $\fii_n(M)=M_n=\bigcup_{i=1}^{n}\{\mu_i\}$, the only things needed for the mapping $\fii_n$ to be well-defined is to check $\bigcup_{i=1}^n\{\fii_n^{-1}(\mu_i)\}=M$ and $\fii_n^{-1}(\mu_i)\cap\fii^{-1}_n(\mu_j)=\emptyset$ for $i\neq j$. For simplicity, we denote the set-valued mapping $\fii_n^{-1}=F_n$ and we will define the mappings $\fii_n,n\in\N$ through defining $F_n:M_n\to 2^M$. Note that if for every $i\in\{1,...,n\}$ holds $\mu_i\in F_n(\mu_i)$, then the mapping $\fii_n$ is a retraction.

In the sequel, by the $n$-tuple $(a_1,a_2,...a_n)$, $a_i\in\R$ we will mean the linear combination $\sum_{i=1}^na_ie_i$ and for a point $x\in X$, $x=\sum_{i=1}^{\infty}x_ie_i$ we will always identify $x$ with $(x_1,x_2,x_3,...)$.
\\\\Set
\begin{center}
\begin{tabular}{l c l}
$\mu_1=0$ & & $F_1(\mu_1)=M$,\\\\
$\mu_2=(1,0)$ & & $F_2(\mu_2)=\{x\in M|\ x_1\geq 1\}$\\
& & $F_2(\mu_1)=M\setminus F_2(\mu_2)$,\\\\
$\mu_3=(-1,0)$ & & $F_3(\mu_3)=\{x\in M|\ x_1\leq -1\}$\\
& & $F_3(\mu_1)=F_2(\mu_1)\setminus F_3(\mu_3)$\\
& & $F_3(\mu_2)=F_2(\mu_2)$.\\
\end{tabular}
\end{center}

It is not difficult to see $\fii_1,\fii_2,\fii_3$ are retractions satisfying the conditions \ref{bed1},\ref{bed3},\ref{bed4} from the Theorem \ref{one} with Lipschitz constant which equals to $\text{uc}(E)\leq K$. Indeed, for $\fii_1$ it is clear as its image is only $\{0\}$. For $\fii_2$, $x,y\in M$ and $i\in\N$ we have
\begin{equation}
\begin{small}
\left|\fii_2(x)_i-\fii_2(y)_i\right|=\begin{cases}
0 & i>1\vee (x_1\geq 1\vee y_1\geq 1)\vee (x_1\leq 0\vee y_1\leq 0),\\
1 & i=1\wedge \left((x_1\geq1\wedge y_1\leq 0)\vee(y_1\geq1\wedge x_1\leq 0)\right),\\
\end{cases}\label{fiien}
\end{small}
\end{equation}
and similarily for $n=3$, $x\in M$ and $i\in\N$ we have
$$\fii_3(x)_i=\begin{cases}
0 & i>1\vee x_1=0,\\
1 & i=1\wedge x_1\geq1,\\
-1 & i=1\wedge x_1\leq -1\\
\end{cases}$$
and therefore for $x,y\in M$
\begin{equation}
\begin{small}
\left|\fii_3(x)_i-\fii_3(y)_i\right|=\begin{cases}
0 & i>1\vee x_1y_1\geq 1\vee x_1=y_1=0,\\
1 & i=1\wedge \left((|x_1|\geq1\wedge y_1=0)\vee(|y_1|\geq1\wedge x_1=0)\right),\\
2 & i=1\wedge x_1y_1\leq -1.\\
\end{cases}\label{fiien}
\end{small}
\end{equation}
Due to the unconditionality of $E$, it is true that for every $x\in X$ and $z\in\R,\ |z|\leq x_1$ holds $\n (z,x_2,x_3,x_4,...)\n\leq \text{uc}(E)\n x\n$. But for every $i\in\N$ the expression in \eqref{fiien} is less or equal to $|x_i-y_i|$, which gives us Lipschitz condition on $\fii_n$ with constant uc$(E)$.

Moreover, the last retraction $\fii_3$ maps $M$ onto the set $C_1^1\subseteq M$ containing all points $x\in M$ with $x=(x_1)$ and $|x_1|\leq 1$. Let us denote $C^d_r=\{x\in M|\ x=(x_1,x_2,...,x_d),|x_i|\leq r,\ i\leq d\}$. From now on, we will proceed inductively. Suppose we have a sequence of retractions $\{\fii_i\}_{i=1}^m$ together with the points $\mu_i$, such that $\fii_m(M)=C_r^r$ and that $\{\fii_i\}_{i=1}^m$ satisfy the conditions \ref{bed1},\ref{bed3},\ref{bed4} from the Theorem \ref{one}. Note that $m=(2r+1)^r$.

We proceed by induction which we divide into two steps. First we find points $\mu_{m+1},...,\mu_{s}$ together with retractions $\fii_{m+1},...,\fii_s$, where $s=(2r+1)^{r+1}$, such that $M_s=C^{r+1}_r$ and such that $\{\fii_i\}_{i=1}^s$ satisfy the conditions \ref{bed1},\ref{bed3},\ref{bed4} from theorem \ref{one}. Then we find points $\mu_{s+1},...,\mu_{t}$ and retractions $\fii_{s+1},...,\fii_t$, where $t=(2r+3)^{r+1}$, $\fii_t:M\to C^{r+1}_{r+1}$ which satisfy \ref{bed1},\ref{bed3},\ref{bed4}. As $\bigcup_{r=1}^{\infty}C_r^r=M$, the condition \ref{bed2} from theorem \ref{one} is obtained as well, which will conclude the proof.
\\
\\On the bounded set $C^r_r$ we define an ordering by the formula
\begin{equation}
\begin{aligned}
(x_1,x_2,...,x_r)\prec (y_1,y_2,...,y_r)\Leftrightarrow\ (x_1> y_1)\vee\\
\exists i\in\{1,...,r-1\}\forall j\in\{1,...,i\}:(x_j=y_j)\wedge (x_{i+1}>y_{i+1}).\label{order}
\end{aligned}
\end{equation}
There exists a bijection $w:\{1,...,(2r+1)^r\}\to C^r_r$, which preserves order.

Let us shorten the notation by introducing indexing functions $a,b$. If $j\in\{1,...,r\}$ and $i\in\{1,...,(2r+1)^r\}$, let $a(j,i)=j(2r+1)^r+i$ and $b(j,i)=(r+j)(2r+1)^r+i$. We set $\mu_{a(j,i)}=(w(i),j)=w(i)+je_{r+1}$ and $\mu_{b(j,i)}=(w(i),-j)=w(i)-je_{r+1}$. Moreover, we formally put $\mu_{a(0,i)}=\mu_{b(0,i)}=w(i)$. Then we define sets
\begin{align*}
F_{a(j,i)}&(\mu_{a(j,i)})=\{x\in F_{a(j,i)-1}(\mu_{a(j-1,i)}), x_{r+1}\geq j\},\\
F_{a(j,i)}&(\mu_{a(j-1,i)})=F_{a(j,i)-1}(\mu_{a(j-1,i)})\setminus F_{a(j,i)}(\mu_{a(j,i)}),\\
F_{a(j,i)}&(\mu_q)=F_{a(j,i)-1}(\mu_q),\ \ q\in\{1,...,a(j,i)-1\}, \mu_q\neq \mu_{a(j-1,i)},
\end{align*}
and
\begin{align*}
F_{b(j,i)}&(\mu_{b(j,i)})=\{x\in F_{b(j,i)-1}(\mu_{b(j-1,i)}), x_{r+1}\leq -j\},\\
F_{b(j,i)}&(\mu_{b(j-1,i)})=F_{b(j,i)-1}(\mu_{b(j-1,i)})\setminus F_{b(j,i)}(\mu_{b(j,i)}),\\
F_{b(j,i)}&(\mu_q)=F_{b(j,i)-1}(\mu_q),\ \ q\in\{1,...,b(j,i)-1\}, \mu_q\neq \mu_{b(j-1,i)}.
\end{align*}
It is easy to see that the formulae above define mappings $\fii_{a(j,i)}$ and $\fii_{b(j,i)}$. Supposed it holds for the mappings $\{\fii_i\}_{i=1}^m$ it is clear that $F_n(\mu_p)\cap F_n(\mu_q)=\emptyset$ for $p\neq q$ and all $n\in\{1,...,s\}$, and that $\mu_n\in F_n(\mu_n)$ and $\bigcup_{i=1}^n F_n(\mu_i)=M$, which means each mapping $\fii_n$ is well-defined and is a retraction onto $M_n$.

Let us check the uniform Lipschitz boundedness. Fix $n\in\{m+1,...,s\}$. Note first that
\begin{align*}
&\forall x=\sum_{i=1}^{\infty}x_ie_i\in X,\forall z\in\l_{\infty}:\\
&\forall i\in\N: 0\leq |z_i|\leq |x_i|\Rightarrow\left\n \sum_{i=1}^{\infty}z_ie_i\right\n\leq \n x\n\cdot\text{uc}(E)
\end{align*}
From this we deduce the Lipschitz boundedness.
\\
\\If $x,y\in M$, then for $i>r+1$ we have $|\fii_n(x)_i-\fii_n(y)_i|=|0-0|=0\leq|x_i-y_i|.$ If $i<r+1$ then we distinguish three cases:
\begin{enumerate}
\item[a)] $|x_i|,|y_i|\leq r$. Then $\fii_n(x)_i=x_i$, $\fii_n(y)_i=y_i$ and therefore $|\fii_n(x)_i-\fii_n(y)_i|=|x_i-y_i|$.
\item[b)] $|x_i|\leq r$, $|y_i|>r$. Then $\fii_n(x)_i=x_i$ and $\fii_n(y)_i=r\sgn(y_i)$. Therefore $|\fii_n(x)_i-\fii_n(y)_i|=|x_i-r\sgn(y_i)|\leq|x_i-y_i|$.
\item[c)] $|x_i|,|y_i|>r$. Then $\fii_n(x)_i=r\sgn(x_i)$, $\fii_n(y)_i=r\sgn(y_i)$ and therefore
$$|\fii_n(x)_i-\fii_n(y)_i|=|r\sgn(x_i)-r\sgn(y_i)|=\begin{cases}
0\leq |x_i-y_i|, & x_iy_i>0,\\
2r\leq |x_i-y_i|, & x_iy_i<0.
\end{cases}$$
\end{enumerate}
Finally, let $i=r+1$. If now $x_iy_i<0$, then either $0\leq\fii_n(x)_i\leq x_i$ and $y_i\leq\fii_n(y)_i\leq 0$ or vice versa. Both options give $|\fii_n(x)_i-\fii_n(y)_i|\leq |x_i-y_i|$, which is what we need.

Let $x_i,y_i\geq 0$. Suppose $n=a(j,k)$ for eligible $j,k$. Then $\fii_n(x)_i=j$ or $\fii_n(x)_i=j-1$ or $\fii_n(x)_i=x_i$, which occurs whenever $0\leq x_i<j-1$. Of course the same holds for $y$. From this we have either $|\fii_n(x)_i-\fii(y)_i|\leq |x_i-y_i|$ or $|\fii_n(x)_i-\fii(y)_i|\leq 1$. If $n=b(j,k)$ for some $j,k$, then $\fii_n(x)_i=r$ whenever $x_i\geq r$ and $\fii_n(x)_i=x_i$ whenever $x_i<r$, the same for $y$. It is clear that $|\fii_n(x)_i-\fii(y)_i|\leq |x_i-y_i|$.

Let $x_i,y_i\leq 0$. If $n=a(j,k)$ for some $j,k$, then $|\fii_n(x)_i-\fii(y)_i|=|0-0|=0\leq |x_i-y_i|$. If $n=b(j,k)$ for some $j,k$, then $\fii_n(x)_i=-j$ or $\fii_n(x)_i=-j+1$ or $\fii_n(x)_i=x_i$, which holds whenever $0\geq x_i>-j+1$. Again, we get either $|\fii_n(x)_i-\fii(y)_i|\leq |x_i-y_i|$ or $|\fii_n(x)_i-\fii(y)_i|\leq 1$.

To sum up all cases, if $x,y\in M$, then either $x_{r+1}=y_{r+1}$ or not. In the first case we have
\begin{equation}
\begin{small}
\begin{aligned}
\n\fii_n(x)-\fii_n(y)\n&=\left\n\sum_{i=1}^{r+1}\left(\fii_n(x)_i-\fii_n(y)_i\right)e_i\right\n\leq\left\n\sum_{i=1}^{r}\left(\fii_n(x)_i-\fii_n(y)_i\right)e_i\right\n+1\\
&\leq \text{uc}(E)\n x-y\n+2\text{bc}(E)\n x-y\n=\\
&=\n x-y\n(\text{uc}(E)+2\text{bc}(E)),\label{lip1}
\end{aligned}
\end{small}
\end{equation}
as $M$ is a $\frac{1}{2\text{bc}(E)}$-separated set, while in the $x_{r+1}\neq y_{r+1}$ case we have
\begin{equation}
\n\fii_n(x)-\fii_n(y)\n=\left\n\sum_{i=1}^{r+1}\left(\fii_n(x)_i-\fii_n(y)_i\right)e_i\right\n\leq\text{uc}(E)\n x-y\n.\label{lip2}
\end{equation}
Considering both cases we get the mapping $\fii_n$ is Lipschitz with constant $K=\text{uc}(E)+2\text{bc}(E)$.
\\
\\It remains to prove that the mappings $\{\fii_n\}_{n=1}^s$ satisfy the commutativity condition \ref{bed4}, provided the mappings $\{\fii_n\}_{n=1}^m$ do. Note that for any $m,n\in\N$, $m\leq n$ holds
\begin{equation}
F_n(\mu_n)\cap F_m(\mu_m)\in\{\emptyset,F_n(\mu_n)\}.
\label{tree}
\end{equation}
Out of this fact the commutativity follows easily: Consider $i<j\in\{1,...,s\}$. First, because $\fii_i$ is a retraction onto $M_i$ and the same holds for $\fii_j$ and $M_j$, from $M_i\subseteq M_j$ follows $\fii_j\fii_i=\fii_i$. It remains to prove $\fii_i\fii_j(x)=\fii_i(x)$ for every $x\in M$.

Take $x\in M$. There exists a maximal finite sequence of indices $1=k_0<...<k_l\leq s$ such that
$$x\in F_{k_l}(\mu_{k_l})\subseteq\cdots\subseteq F_{k_0}(\mu_{k_0}).$$
Clearly if $c(i)$ is the biggest index such that $k_{c(i)}\leq i$, then $\fii_i(F_{k_d}(\mu_{k_d}))=\mu_{c(i)}$ for all $d,\ c(i)\leq d\leq l$. This applies analogously for $\fii_j$ with $c(j)$. From the fact that both $x,\mu_{k_{c(j)}}\in F_{k_{c(j)}}(\mu_{k_{c(j)}})\subseteq F_{k_{c(i)}}(\mu_{k_{c(i)}})$ we get simply
$$\fii_i\fii_j(x)=\fii_i(\mu_{k_{c(j)}})=\mu_{k_{c(i)}}=\fii_i(x),$$
which finishes the proof of commutativity.
\\\\
To finish the proof, it remains to show the construction of retractions $\fii_{s+1},...,\fii_t$, where $t=(2r+3)^{r+1}$, $\fii_t:M\to C^{r+1}_{r+1}$ which satisfy \ref{bed1},\ref{bed3},\ref{bed4}.
\\\\
For $i\in\N$ let us define an $i$-predecessor function $p_i:M\to M$ by
$$p_i\left(\sum_{n=1}^{\infty}x_ne_n\right)=\sum_{n=1}^{\infty}x_ne_n-\sgn(x_i)e_i.$$
Now for every $j\in\{1,...,r+1\}$ we introduce sets
\begin{align*}
A_{j,1}=&\{(x_1,...,x_{j-1},r+1,x_{j+1},...,x_{r+1}):\right.\\
&\left.x_i\in\Z\wedge|x_i|\leq r+1\ \text{for}\ i< j\wedge |x_i|\leq r\ \text{for}\ i>j \},\\
A_{j,-1}=&\{(x_1,...,x_{j-1},-r-1,x_{j+1},...,x_{r+1}):\right.\\
&\left.x_i\in\Z\wedge|x_i|\leq r+1\ \text{for}\ i<j\wedge |x_i|\leq r\ \text{for}\ i>j \}.
\end{align*}
Clearly, $A_{j,-1},A_{j,1}\sub C_{r+1}^{r+1}$ and $|A_{j,-1}|=|A_{j,1}|=(2r+1)^{r+1-j}(2r+3)^{j-1}$. Moreover,
$$\bigcup_{\substack{
            j\in\{1,...,r+1\}\\
            i\in\{-1,1\}}}
     A_{j,i}=C^{r+1}_{r+1}\setminus C_r^{r+1}$$
and it is a disjoint union. For each $j$, choose any bijection $w_j:\{1,...,|A_{j,1}|\}\to A_{j,1}$ and fix it. Define $\ov w_j:\{1,...,|A_{j,1}|\}\to A_{j,-1}$, by $\ov w_j(i)=(w_j(i)_1,w_j(i)_2,...,-w_j(i)_j,...,w_j(i)_{r+1})$. For simplicity, for $j\in\{1,...,r+1\},i\in\{1,...,|A_{j,1}|\}$ put
$$\a(j,i)=s+2\sum_{k=1}^{j-1}|A_{k,1}|+i,\ \ \b(j,i)=s+2\sum_{k=1}^{j-1}|A_{k,1}|+|A_{j,1}|+i.$$
Then we finally set $\mu_{\a(j,i)}=w_j(i),\ \mu_{\b(j,i)}=\ov w_j(i).$ Now we define mappings $\{F_n\}_{n=s+1}^t$ via
\begin{align*}
F_{\a(j,i)}&(\mu_{\a(j,i)})=\{x\in F_{\a(j,i)-1}(p_j(\mu_{\a(j,i)})), x_j\geq r+1\},\\
F_{\a(j,i)}&(p_j(\mu_{\a(j,i)}))=F_{\a(j,i)-1}(p_j(\mu_{\a(j,i)}))\setminus F_{\a(j,i)}(\mu_{\a(j,i)}),\\
F_{\a(j,i)}&(\mu_q)=F_{\a(j,i)-1}(\mu_q),\ \ q\in\{1,...,\a(j,i)-1\}, \mu_q\neq p_j(\mu_{\a(j,i)}),
\end{align*}
and
\begin{align*}
F_{\b(j,i)}&(\mu_{\b(j,i)})=\{x\in F_{\b(j,i)-1}(p_j(\mu_{\b(j,i)})), x_{j}\leq -r-1\},\\
F_{\b(j,i)}&(p_j(\mu_{\b(j,i)}))=F_{\b(j,i)-1}(p_j(\mu_{\b(j,i)}))\setminus F_{\b(j,i)}(\mu_{\b(j,i)}),\\
F_{\b(j,i)}&(\mu_q)=F_{\b(j,i)-1}(\mu_q),\ \ q\in\{1,...,\b(j,i)-1\}, \mu_q\neq p_j(\mu_{\b(j,i)}).
\end{align*}
Obviously, the upper equations define mappings $\fii_{\a(j,i)}$ and $\fii_{\b(j,i)}$ for all $j\in\{1,...,r+1\}$ and $i\in\{1,...,|A_{j,1}|\}$, hence the mappings $\{\fii_n\}_{n=s+1}^{t}$ are well-defined and it is an easy check that each such $\fii_n$ is a retraction onto the set $M_n$.

Note that the sets $\{F_n(\mu_n)\}_{n=1}^t$ still satisfy the condition \eqref{tree} so the commutativity condition \ref{bed4} from theorem \ref{one} is obtained similarly as it was done for retractions $\{\fii_n\}_{n=1}^s$.

It remains to show the mappings are Lipschitz-bounded. Let us for simplicity denote $\b_k=\b(k-1,|A_{k-1,1}|)$ for $1<k\leq r+1$ and $\b_1=s$, the index of first retraction $\fii_{\b_k}$ such that $A_{k-1,-1}\sub M_{\b_k}$. Fix $n\in\{s+1,...,t\}$. We will prove that there exists at most one $j=j(n)\in\N$ such that for all $l\in\N$, $l\neq j$ and all $x,y\in M$ we have $|\fii_n(x)_l-\fii_n(y)_l|\leq |x_l-y_l|$ out of which the Lipschitz boundedness of $\fii_n$ follows.
If $n=\a(j,i)$ for some eligible $j,i$, then for every $x\in M$ holds
$$\fii_n(x)_l=\begin{cases}
0 & l>r+1,\\
x_l & (l\leq r+1,|x_l|\leq r)\vee(l<j, |x_l|= r+1),\\
r\sgn(x_l) & (j<l\leq r+1, |x_l|>r)\vee(j=l,x_l<-r)\vee\\
& \left(j=l,x_l> r,\forall\mu\in M_n: \fii_{\b_j}(x)_j\neq p_j(\mu)\right),\\
(r+1)\sgn(x_l) & (l<j, |x_l|> r+1)\vee\\
& \left(l=j,x_l\geq r+1,\exists\mu\in M_n: \fii_{\b_j}(x)_j=p_j(\mu)\right),
\end{cases}$$
while if $n=\b(j,i)$ for some $j,i$, then for every $x\in M$ we have
$$\fii_n(x)_l=\begin{cases}
0 & l>r+1,\\
x_l & (l\leq r+1,|x_l|\leq r)\vee(l<j, |x_l|=r+1)\vee\\
&(l=j,x_l=r+1),\\
r\sgn(x_l) &   (j<l\leq r+1, |x_l|>r)\vee\\
& \left(j=l,x_l< -r,\forall\mu\in M_n: \fii_{\b_j}(x)_j\neq p_j(\mu)\right),\\
(r+1)\sgn(x_l) & (l<j, |x_l|> r+1)\vee(l=j,x_l> r+1)\vee\\
& \left(l=j,x_l\leq -r-1,\exists\mu\in M_n: \fii_{\b_j}(x)_j=p_j(\mu)\right).
\end{cases}$$
If $x,y\in M$, it is not difficult to see that if $|\fii_n(x)_l-\fii_n(y)_l|>|x_l-y_l|$, then $l=j$ and $\fii_n(x)_l=(r+1)\sgn(x_l)$, $\fii_n(y)_l=r\sgn(y_l)$ or vice versa and $x_ly_l>0$. Particularly $|\fii_n(x)_l-\fii_n(y)_l|=1$ and $|x_l-y_l|=0$. For all other $l$, i.e. $l\neq j,l\in\N$ holds $|\fii_n(x)_l-\fii_n(y)_l|\leq|x_l-y_l|$, which is what we need.

Therefore we get by computation similar to those done in \eqref{lip1} and \eqref{lip2} that $\fii_n$ is a Lipschitz mapping with constant $K=\text{uc}(E)+2\text{bc}(E)$, which concludes the induction.

As $\bigcup_{r=1}^\infty C^{r}_r=M$ the condition \ref{bed2} from theorem \ref{one} is also satisfied and hence our proof is finished.
\end{proof}
\begin{remark*}
In \eqref{order} it was not necessary for our construction to choose exactly this order. In fact, any bijection $w:\{1,...,(2r+1)^r\}\to C_r^r$ would suit our purpose. We chose the order \eqref{order} for simplicity. In this case we have $\mu_{a(j-1,i)}=p_j(\mu_{a(j,i)})$ and $\mu_{b(j-1,i)}=p_j(\mu_{b(j,i)})$ for $p_j$ the $j$-predecessor function and $i\in\{1,...,(2r+1)^r\}$, $j\in\{1,...,r\}$.
\end{remark*}
\begin{corollary} If $E=\{e_i\}_{i=1}^{\infty}$ denotes the canonical basis in $c_0$ and $M=M(E)\sub c_0$ the integer grid, then the Free-space $\F(M)$ has a monotone Schauder basis.
\end{corollary}
\begin{proof}
applying the construction of the retractions from the lemma \ref{dva} to $(c_0,E)$, we get Lipschitz constant $K=1$, (see estimates \eqref{lip1} and \eqref{lip2}). Therefore, $\F(M)$ has a monotone Schauder basis.
\end{proof}
\begin{corollary} Let $\mathcal N\sub c_0$ be a net. Then the Free-space $\F(\mathcal N)$ has a Schauder basis.
\end{corollary}
\begin{proof}
If we use the notation from previous corrolary, $M$ is a $(1,1)$-net in $c_0$. But as all nets in an infinite-dimensional space are Lipschitz equivalent (\cite{benlin}, p.239, Proposition 10.22), $\mathcal N$ is Lipschitz equivalent to the grid $M$ and therefore $\F(\mathcal N)$ is isomorphic to $\F(M)$, which concludes the proof.
\end{proof}
\begin{corollary}
Let  $\mathcal N$ be a net in any of the
following metric spaces: $C(K)$, $K$ metrizable compact, or
$c_0^+$ (the subset of $c_0$ consisting of elements
with non-negative coordinates).
Then $\f(\mathcal N)$ has a Schauder basis.
\end{corollary}
\begin{proof}
Follows immediately from Theorem \ref{ck}.
\end{proof}
\begin{corollary}
Let  $\mathcal N\sub\R^n$ be a net. Then $\f(\mathcal N)$ has a Schauder basis.
\end{corollary}
\begin{proof}
It follows from the proof of lemma \ref{dva} that $\F(\Z^n)$ has a Schauder basis and $\F(\Z^n)\cong\F(\mathcal N)$ by Proposition \ref{sam-f}, which gives the result.
\end{proof}

\end{document}